\documentclass[12pt]{amsart}
\usepackage{amssymb}
\usepackage{amsfonts}
\usepackage{graphicx}
\usepackage{epsfig}
\usepackage{amsmath}

\theoremstyle{plain}
\newtheorem{theorem}{Theorem}[section]

\newtheorem{corollary}[theorem]{Corollary}
\theoremstyle{definition}

\newtheorem{remark}[theorem]{Remark}

\def\<{\langle}
\def\>{\rangle}

\begin{document}
\title
        {On the recurrence formula of the Euler zeta functions}
        \author{Joonhyung Kim}

\address{Joonhyung Kim \\
    Department of Mathematics Education, Hannam University, 70 Hannam-ro, Daedeok-gu, Daejeon 306-791, Republic of Korea}
\email{calvary\char`\@hnu.kr}%
        \date{}
        \maketitle

\begin{abstract}
In this paper, we find a new recurrence formula of the Euler zeta functions.
\end{abstract}
\footnotetext[3]{\sl{This work was supported by 2015 Hannam University Research Fund.}}

\section{Introduction}
The Euler zeta function is defined as $\displaystyle{\zeta_E(s)=\sum^{\infty}_{n=1}\frac{(-1)^{n-1}}{n^s}}$ for $s \in \mathbb{C}$. In \cite{LR}, Lee-Ryoo found the following recurrence formula of $\zeta_E(2s)$ for $s \in \mathbb{N}$ using Fourier series.

\begin{theorem}(Theorem 4 of \cite{LR})
For $s \ge 2(\in \mathbb{N})$ and $\zeta_E(2)=\frac{\pi^2}{12}$,
\begin{align*}
\zeta_E(2s) & =\frac{(-1)^s(2\pi)^{2s}}{_{2s}P_{2s-1}}\{\frac{1}{2^{2s+1}}\frac{2^{2s+1}-12s^2+3}{(2s-1)(2s+3)}\\ & -\sum^{s-1}_{k=1}(-1)^k\frac{1}{(2\pi)^{2k}}\zeta_E(2k)(_{2s}P_{2k-1}-_{2s-2}P_{2k-1}) \}.
\end{align*}
\end{theorem}

The proof is very elementary. They first considered a function $f(x)=x^{2m}$ for $-2<x<2$ and found Fourier coefficients. Then $f(x)$ can be written as
$$
f(x)=\frac{2^{2m}}{2m+1}+\sum^{\infty}_{n=1}a_n\cos\frac{n\pi x}{2},
$$
where $\displaystyle{a_n=\sum^m_{k=1}(-1)^{k+1}}$$\displaystyle{_{2m}P_{2k-1}2^{2m-2k+1}\frac{2^{2k}}{n^{2k}\pi^{2k}}\cos n\pi}$.\\
By substituting $x=1$, they got
$\displaystyle{\sum^s_{k=1}(-1)^{k}}$$\displaystyle{_{2s}P_{2k-1}\frac{1}{2^{2k}\pi^{2k}}\zeta_E(2k)=\frac{2s+1-2^{2s}}{(2s+1)(2^{2s+1})}}$.
In the case of $m=1$, one easily gets $\displaystyle{\zeta_E(2)=\frac{\pi^2}{12}}$. Using above equation, it is easy to get the formula in the theorem. See \cite{LR} for more details.\\
The goal of this article is to get more refined version of the recurrence formula of the Euler zeta functions.

\section{The Recurrence Formula}

The following is the main result.

\begin{theorem}
For $s \ge 2(\in \mathbb{N})$ and $\zeta_E(2)=\frac{\pi^2}{12}$,
$$
\zeta_E(2s)=\frac{(-1)^s(\pi)^{2s}}{_{2s}P_{2s-1}}\{\frac{1}{(2s-1)(2s+1)}-\sum^{s-1}_{k=1}(-1)^k\frac{1}{\pi^{2k}}\zeta_E(2k)(_{2s}P_{2k-1}-_{2s-2}P_{2k-1}) \}.
$$
\end{theorem}

\begin{proof}
We'll use the same function $f(x)=x^{2m}$ for $-2<x<2$. As in the proof of \cite{LR},
$$
f(x)=\frac{2^{2m}}{2m+1}+\sum^{\infty}_{n=1}a_n\cos\frac{n\pi x}{2},
$$
where $\displaystyle{a_n=\sum^m_{k=1}(-1)^{k+1}}$$\displaystyle{_{2m}P_{2k-1}2^{2m-2k+1}\frac{2^{2k}}{n^{2k}\pi^{2k}}\cos n\pi}$.\\
Now we substitute $x=0$. Then,
\begin{align*}
0 & =\frac{2^{2m}}{2m+1}+\sum^{\infty}_{n=1}\sum^{m}_{k=1}(-1)^{k+1}(_{2m}P_{2k-1})2^{2m-2k+1}\frac{2^{2k}}{n^{2k}\pi^{2k}}\cos n\pi\\ 
& = \frac{2^{2m}}{2m+1}+\sum^{\infty}_{n=1}\sum^{m}_{k=1}(-1)^{k+1}(_{2m}P_{2k-1})2^{2m+1}\frac{(-1)^n}{n^{2k}\pi^{2k}}\\
& = \frac{2^{2m}}{2m+1}+\sum^{m}_{k=1}(-1)^k(_{2m}P_{2k-1})2^{2m+1}\frac{1}{\pi^{2k}}\zeta_E(2k).
\end{align*}
Therefore $\displaystyle{\sum^s_{k=1}(-1)^{k}}$$\displaystyle{_{2s}P_{2k-1}\frac{1}{\pi^{2k}}\zeta_E(2k)=-\frac{1}{2(2s+1)}}$ and hence we also get  $\displaystyle{\sum^{s-1}_{k=1}(-1)^{k}}$$\displaystyle{_{2s-2}P_{2k-1}\frac{1}{\pi^{2k}}\zeta_E(2k)=-\frac{1}{2(2s-1)}}$. By subtracting the second equation from the first equation, we get the result.
\end{proof}

\begin{corollary}
For $s \ge 2(\in \mathbb{N})$ and $\zeta_E(2)=\frac{\pi^2}{12}$,
$$
\zeta_E(2s)=\frac{(-1)^s\pi^{2s}}{2s-1}\{\frac{1}{(2s+1)!}-\frac{1}{s}\sum^{s-1}_{k=1}(-1)^k\frac{1}{\pi^{2k}}\zeta_E(2k)\frac{(2k-1)(2s-k)}{(2s-2k+1)!} \}.
$$
\end{corollary}
\begin{proof}
Since $\displaystyle{_{2s}P_{2s-1}=(2s)!}$ and $\displaystyle{_{2s}P_{2k-1}-_{2s-2}P_{2k-1}=\frac{2(2s-2)!(2k-1)(2s-k)}{(2s-2k+1)!}}$, the corollary immediately follows.
\end{proof}

\begin{remark}
One can get many such recurrence formulas by considering different functions. Even the same function, one can also get many formulas by substituting different values.
\end{remark}


\end{document}